\documentclass[11pt]{article}
\usepackage[margin=1.0in]{geometry}

\pdfoutput=1 

\usepackage{amssymb,amsmath,amstext,amsthm}
\usepackage{bbm}
\usepackage{times}
\usepackage{graphicx} 
\usepackage{subfigure} 
\usepackage{hyperref}
\usepackage{algorithm}
\usepackage{algorithmic}

\newtheorem{definition}{Definition}
\newtheorem{theorem}{Theorem}
\newtheorem{lemma}{Lemma}

\DeclareMathOperator{\sthat}{\textrm{s.t.}}

\DeclareMathOperator{\argmin}{argmin}

\newcommand{\setR}{\mathbb{R}}
\newcommand{\ignore}[1]{}
\newcommand{\indFunc}{\mathbbm{1}}
\newcommand{\norm}[1]{\left\Vert#1\right\Vert}

\begin{document}

\title{Distributed Convex Optimization with Many Non-Linear
  Constraints}
%
%
\author{Joachim Giesen 
\and
S\"oren Laue~\footnote{Friedrich-Schiller-University Jena, Germany, 
\{joachim.giesen, soeren.laue\}@uni-jena.de}}

\date{}
%
%
%
\maketitle              

\begin{abstract}
  We address the problem of solving convex optimization problems with
  many non-linear constraints in a distributed setting. Our approach
  is based on an extension of the alternating direction method of
  multipliers (ADMM). Although it has been invented decades ago, ADMM
  so far can be applied only to unconstrained problems and problems
  with linear equality or inequality constraints. Our extension can
  directly handle arbitrary inequality constraints. It combines the
  ability of ADMM to solve convex optimization problems in a
  distributed setting with the ability of the Augmented Lagrangian
  method to solve constrained optimization problems, and as we show,
  it inherits the convergence guarantees of both ADMM and the
  Augmented Lagrangian method.
\end{abstract}

\section{Introduction}

The increasing availability of distributed hardware suggests
addressing large scale optimization problems by distributed
algorithms. Large scale optimization problems involve a large number
of \emph{optimization variables}, or a large number of \emph{input
  parameters}, or a large number of \emph{constraints}. Here we
address the latter case of a large number of constraints.

In recent years, the alternating direction method of multipliers
(ADMM) that was proposed by Glowinski and Marroco~\cite{Glowinski1975}
and by Gabay and Mercier~\cite{Gabay1976} already decades ago obtained
considerable attention, also beyond the machine learning community,
because it allows to solve convex optimization problems that involve a
large number of parameters in a distributed
setting~\cite{boydADMM}. For instance, the parameters in ordinary
least squares regression are just the data points. The optimization
problem behind a typical machine learning method usually aims for
minimizing a loss-function that is the sum of the losses for each data
point. Hence, the objective function $f$ of such problems is
separable, i.e., it holds that $f(x) = \sum_i f_i(x_i)$, where $f_i$
is determined by the $i$-th data point. In this case ADMM lends itself
to a distributed implementation where the data points are distributed
on different compute nodes.

Standard ADMM works for unconstrained optimization problems and for
optimization problems with linear equality and/or inequality
constraints. Surprisingly, so far no general convex inequality
constraints have been considered directly in the context of ADMM.
Optimization problems with a large number of constraints typically
also arise as big data problems, but instead of contributing a term to
the objective function, each data point now contributes a constraint
to the problem. An illustrative example is the core vector
machine~\cite{TsangKK07,Yildirim08}, where the smallest enclosing ball
for a given set of data points has to be computed.  The objective
function here is the radius of the ball that needs to be minimized,
and every data point contributes a non-linear constraint, namely the
distance of the point from the center must be at most the
radius. Another example problem are robust SVMs that we discuss in
more detail later in the paper.

In principle, standard ADMM can also be used for solving constrained
optimization problems. A distributed implementation of the
straightforward extension of ADMM leads to non-trivial constrained
optimization subproblems that have to be solved in every
iteration. Solving constrained problems is typically transformed into
a sequence of unconstrained problems. Hence, this approach features
three nested loops, the outer loop for reaching consensus, one loop
for the constraints, and an inner loop for solving unconstrained
problems. Alternatively, one could use the standard Augmented
Lagrangian method, originally known as the method of
multipliers\cite{Hestenes69}, that has been specifically designed for
solving constrained optimization problems. Combining the Augmented
Lagrangian method with ADMM allows to solve general constrained
problems in a distributed fashion by running the Augmented Lagrangian
method in an outer loop and ADMM in an inner loop. Again, we end up
with three nested loops, the outer loop for the augmented Lagrangian
method and the standard two nested inner loops for ADMM.  Thus, one
could assume that any distributed solver for constrained optimization
problems needs at least three nested loops: one for reaching
consensus, one for the constraints, and one for the unconstrained
problems. The key contribution of our paper is showing that this is
not the case. One of the nested loops can be avoided by merging the
loops for reaching consensus and dealing with the constraints. Our
approach, that only needs two nested loops, combines ADMM with the
Augmented Lagrangian method differently than the direct approach of
running the Augmented Lagrangian method in the outer and ADMM in the
inner loop. But the latter combination still provides us with a good
baseline to compare against.

\paragraph{Related work.}

To the best of our knowledge, our extension of ADMM is the first
distributed algorithm for solving general convex optimization problems
with no restrictions on the type of constraints or assumptions on the
structure of the problem. Surprisingly, even our baseline method of
running the augmented Lagrangian method in an outer loop and ADMM in
an inner loop has not been studied before. The only special case, that
we are aware of, are quadratically constrained quadratic problems that
have been addressed by Huang and Sidiropoulos~\cite{Huang2016} using
consensus ADMM. However, their approach does not scale to many
constraints, because every constraint gives rise to a new subproblem.

Mosk{-}Aoyama et al.~\cite{distOpt2010} have designed and analyzed a
distributed algorithm for solving convex optimization problems with
separable objective function and linear equality constraints. Their
algorithm blends a gossip-based information spreading, iterative
gradient ascent method with the barrier method from interior-point
algorithms. It is similar to ADMM and can also handle only linear
constraints.

Zhu and Mart\'inez~\cite{ZhuM2012} have introduced a distributed
multi-agent algorithm for minimizing a convex function that is the sum
of local functions subject to a global equality or inequality
constraint. Their algorithm involves projections onto local constraint
sets that are usually as hard to compute as solving the original
problem with general constraints. For instance, it is well known via
standard duality theory that the feasibility problem for linear
programs is as hard as solving linear programs. This holds true for
general convex optimization problems with vanishing duality gap.

In principle, the standard ADMM can also handle convex constraints by
transforming them into indicator functions that are added to the
objective function. However, this leads to subproblems that need to be
solved in each iteration that entail computing a projection onto the
feasible region. This entails the same issues as the method by Zhu and
Mart\'ines~\cite{ZhuM2012} since computing these projections can be as
hard as solving the original problem. We will elaborate on this in
more detail in Section~\ref{sec:problem}.

The recent literature on ADMM is vast. Most papers on ADMM stay in the
standard framework of optimizing a function or a sum of functions
subject to linear constraints and make contributions to one or more of
the following aspects: (1) Theoretical (and practical) convergence
guarantees~\cite{deng2016,he2012,he2015,hong2012,nishihara2015}, (2)
convergence guarantees for asynchronous ADMM~\cite{zhang2014}, (3)
splitting the problem into more than two
subproblems~\cite{chen2016,lin2015}, (4) optimal penalty parameter
selection~\cite{GhadimiTSJ15}, (5) solving the individual subproblems
efficiently or inexactly while still guaranteeing
convergence~\cite{chang2015,chen2017,Lin2011,ScheinbergMG10}, and (6)
applications of ADMM.

%

\section{Alternating direction method of multipliers}
\label{sec:admm}

Here, we briefly review the alternating direction method of
multipliers (ADMM) and discuss how it can be adapted to deal with
distributed data.  ADMM is an iterative algorithm that in its most
general form can solve convex optimization problems of the form
\begin{equation} \label{eq:ADMM}
  \begin{array}{cl}
    \min_{x, z} & f_1(x) + f_2(z) \\
    \sthat & Ax + Bz - c = 0,
  \end{array}
\end{equation}
where $f_1:\setR^{n_1}\to\setR\cup \{\infty\}$ and
$f_2:\setR^{n_2}\to\setR\cup\{\infty\}$ are convex functions,
$A\in\setR^{m\times n_1}$ and $B\in\setR^{m\times n_2}$ are matrices,
and $c\in\setR^m$.

ADMM can obviously deal with linear equality constraints, but it can
also handle linear inequality constraints. The latter are reduced to
linear equality constraints by replacing constraints of the form $Ax
\leq b$ by $Ax + s = b$, adding the slack variable $s$ to the set of
optimization variables, and setting $f_2(s) = \indFunc_{\setR_+^m}(s)$,
where
\[
\indFunc_{\setR_+^m} (s) =
\begin{cases}
  0,& \text{if } s\geq 0\\ \infty, & \text{otherwise},
\end{cases}
\]
is the indicator function of the set $\setR_+^m = \{x\in\setR^m|x\geq
0\}$. Note that $f_1$ and $f_2$ are allowed to take the value
$\infty$.

Recently, ADMM regained a lot of attention, because it allows to solve
problems with separable objective function in a distributed
setting. Such problems are typically given as
\[
\begin{array}{cl}
  \min_{x} & \sum_i f_i(x),
\end{array}
\]
where $f_i$ corresponds to the $i$-th data point (or more generally
$i$-th data block) and $x$ is a weight vector that describes the data
model. This problem can be transformed into an equivalent optimization
problem, with individual weight vectors $x_i$ for each data point
(data block) that are coupled through an equality constraint,
\[
  \begin{array}{cl}
    \min_{x_i, z} & \sum_i f_i(x_i) \\
    \sthat & x_i - z = 0 \quad \forall \, i,
  \end{array}
\]
which is a special case of Problem~\ref{eq:ADMM} that can be solved by
ADMM in a distributed setting by distributing the data.

Adding convex inequality constraints to Problem~\ref{eq:ADMM} does not
destroy convexity of the problem, but so far ADMM cannot deal with
such constraints. Note that the problem only remains convex, if all
equality constraints are induced by affine functions. That is, we
cannot add convex equality constraints in general without destroying
convexity.

Our goal for the following sections is extending ADMM such that it can
also deal with nonlinear, convex inequality constraints. For problems
with many constraints we will show that these constraints can be
distributed similarly as the data points in problems with separable
objective function are distributed for the standard ADMM.

\section{Problems with non-linear constraints}
\label{sec:problem}

We consider convex optimization problems of the form
\begin{equation} \label{eq:general} 
  \begin{array}{cl}
    \min_{x, z} & f_1(x) + f_2(z) \\
    \sthat & g_0(x) \leq 0 \\
    & h_1(x) + h_2(z) = 0,
  \end{array}
\end{equation}
where $f_1$ and $ f_2$ are as in Problem~\ref{eq:ADMM},
$g_0:\setR^{n_1}\to\setR^{p}$ is convex in every component, and
$h_1:\setR^{n_1}\to\setR^{m}$ and $h_2:\setR^{n_2}\to\setR^{m}$ are
affine functions. In the following we assume that the problem is
feasible, i.e., that a feasible solution exists, and that strong
duality holds. A sufficient condition for strong duality is that the
interior of the feasible region is non-empty. This condition is known
as Slater's condition for convex optimization
problems~\cite{Slater50}.

As stated before, our goal is extending ADMM such that it can also
solve Problem~\ref{eq:general}. The simple trick of adding
non-negative slack variables only works, if the constraints
$g_0(x)\leq 0$ are affine. Still, this trick gives some insight into
the general problem.  We have dealt with the non-negativity
constraints on the slack variables by adding an indicator function to
the objective function. The indicator function forces ADMM to project
the solution in every iteration onto the set $\{ s\in\setR^m| s\geq 0
\}$ which is just the non-negative orthant. Projecting onto the
non-negative orthant is an easy problem and thus ADMM can efficiently
deal with linear inequality constraints. As we have already mentioned
in the introduction the idea of transforming the
constraints into indicator functions and adding them to the objective
function can be generalized to non-linear constraints. However, ADMM
then needs to compute 
in every iteration a projection onto the more complicated feasible set
$\{x\,|\,g_0(x)\leq 0\}$. Such a projection is the solution to the
following constrained optimization problem
\[
  \begin{array}{cl}
	\min_x & \|x\|^2 \\
	\sthat &  g_0(x)\leq 0,
  \end{array}
\]
whose solving, depending on the constraints, requires a QP, SOCP or
even SDP solver. Thus we have only deferred the difficulties that have
been induced by the non-linear constraints to the subproblem of
computing the projections. Here, we will devise a method for dealing
with arbitrary constraints directly without any hard-to-compute
projections.

\section{ADMM extension}
\label{sec:extension}

For our extension of ADMM and its convergence analysis we need to work
with an equivalent reformulation of Problem~\ref{eq:general}, where we
replace $g_0(x)$ by
\[
g(x) = \max\{0,\, g_0(x)\}^2,
\]
with componentwise maximum, and turn the convex inequality constraints
into convex equality constraints. Thus, in the following we 
consider optimization problems of the form
\begin{equation} \label{eq:general2} 
  \begin{array}{cl}
    \min_{x, z} & f_1(x) + f_2(z) \\
    \sthat & g(x) = 0 \\
    & h_1(x) + h_2(z) = 0,
  \end{array}
\end{equation}
where $g(x) = \max\{0,\, g_0(x)\}^2$, which by construction is again
convex in every component and differentiable if $g(x)$ is differentiable. 
Note, though, that the constraint $g(x)= 0$
is no longer affine. However, we show in the following that
Problem~\ref{eq:general2} can still be solved efficiently.

Analogously to ADMM our extension builds on the Augmented
Lagrangian for Problem~\ref{eq:general2} which is the following
function
\begin{align*}
  L_\rho (x, z, \mu, \lambda)
  &= f_1(x) + f_2(z) + \frac{\rho}{2}
  \norm{g(x)}^2 + \mu^\top g(x) \\
  &\quad + \frac{\rho}{2} \norm{h_{1}(x) + h_{2}(z)}^2 +
  \lambda^\top \left(h_{1}(x) + h_{2}(z)\right),
\end{align*}
where $\mu\in\setR^p$ and $\lambda\in\setR^m$ are Lagrange
multipliers, $\rho >0$ is some constant, and $\norm{\cdot}$ denotes
the Euclidean norm. The Lagrange multipliers are also referred to as
dual variables.

Algorithm~\ref{algo:1} is our extension of ADMM for solving instances
of Problem~\ref{eq:general2}. It runs in iterations. In the $(k+1)$-th
iteration the primal variables $x^k$ and $z^k$ as well as the dual
variables $\mu^k$ and $\lambda^k$ are updated.

\begin{algorithm}[h!]
  \caption{ADMM for problems with non-linear constraints}
  \label{algo:1}
  \begin{algorithmic}[1]
    \STATE {\bfseries input:} instance of Problem~\ref{eq:general2}
    \STATE {\bfseries output:} approximate solution $x\in\setR^{n_1},
    z\in\setR^{n_2}, \mu\in\setR^{p}, \lambda\in\setR^{m}$ 
    \STATE initialize $x^0 = 0$, $z^0 = 0$, $\mu^0 = 0$, $\lambda^0 = 0$, and $\rho$ to some constant $>0$
    \REPEAT
    \STATE  $x^{k+1} :=\quad \argmin_{x}\, L_\rho(x, z^k, \mu^k, \lambda^k)$ \label{algo:x}
    \STATE  $z^{k+1} :=\quad \argmin_{z}\, L_\rho(x^{k+1}, z, \mu^k, \lambda^k)$ \label{algo:z}
    \STATE $\mu^{k+1} :=\quad  \mu^k + \rho g(x^{k+1})$ \label{algo:mu}	
    \STATE $\lambda^{k+1} :=\quad  \lambda^k + \rho\left(h_1(x^{k+1}) + h_2(z^{k+1})\right)$ \label{algo:lambda}
    \UNTIL{convergence}
    \RETURN $x^k, z^k, \mu^k, \lambda^k$
  \end{algorithmic}
\end{algorithm}

\section{Convergence analysis}
\label{sec:convergence}

From duality theory we know that for all $x\in\setR^{n_1}$ and
$z\in\setR^{n_2}$
\begin{equation} \label{eq:assumption}
L_0(x^*, z^*, \mu^*, \lambda^*) \leq L_0(x, z, \mu^*, \lambda^*),
\end{equation}
where $L_0$ is the Lagrangian of Problem~\ref{eq:general2} and $x^*,
z^*, \mu^*$, and $\lambda^*$ are optimal primal and dual
variables. Note, that $x^*, z^*, \mu^*$, and $\lambda^*$ are not
necessarily unique. Here, they refer just to one optimal solution.
Also note that the Lagrangian is identical to the Augmented Lagrangian
with $\rho = 0$. Given that strong duality holds, the optimal solution
to the original Problem~\ref{eq:general2} is identical to the optimal
solution of the Lagrangian dual.

We need a few more definitions.  Let $f^k = f_1(x^k) + f_2(z^k)$ be
the objective function value at the $k$-th iterate $(x^k,z^k)$ and let
$f^*$ be the optimal function value. Let $r_g^k = g(x^k)$ be the
residual of the nonlinear equality constraints, i.e., the constraints
originating from the convex inequality constraints, and let $r_h^k =
h_1(x^k) + h_2(z^k)$ be the residual of the linear equality
constraints in iteration $k$.

Our goal in this section is to prove the following theorem. 

\begin{theorem}\label{thm:1}
When Algorithm~\ref{algo:1} is applied to an instance of
Problem~\ref{eq:general2}, then
\[
\lim_{k\to\infty} r_g^k = 0, \quad \lim_{k\to\infty} r_h^k = 0, \quad
\textrm{ and }\quad \lim_{k\to\infty} f^k = f^*.
\]
\end{theorem}

The theorem states primal feasibility and convergence of the primal
objective function value. Note, however, that convergence to primal
optimal points $x^*$ and $z^*$ cannot be guaranteed. This is the case
for the original ADMM as well. Additional assumptions on the problem,
like, for instance, a unique optimum, are
necessary to guarantee convergence to the primal optimal
points. However, the points $x^k, z^k$ will be primal optimal and
feasible up to an arbitrarily small error for sufficiently large $k$.

The proof of Theorem~\ref{thm:1} follows along the lines of the
convergence proof for the original ADMM in~\cite{boydADMM} and is
subdivided into four lemmas.

\begin{lemma} \label{lem:0}
  The dual variables $\mu^k$ are non-negative for all iterations,
  i.e., it holds that $\mu^k \geq 0$ for all $k\in\mathbb{N}$.
\end{lemma}
\begin{proof}
  The proof is by induction. In Line~3 of Algorithm~\ref{algo:1} the
  dual variable is initialized as $\mu^0=0$. If $\mu^k \geq 0$, then
  it follows from the update rule in Line~7 of Algorithm~\ref{algo:1}
  that
  \[
  \mu^{k+1} = \mu^k + \rho g(x) \geq 0,
  \]
  since $g(x) = \max\{0,\, g_0(x)\}^2 \geq 0$ and by assumption also
  $\rho >0$.
\end{proof}


\begin{lemma} \label{lem:1}
  The difference between the optimal objective function value and its
  value at the $(k+1)$-th iterate can be bounded as
  \[
  f^* - f^{k+1} \leq (\mu^*)^\top r_g^{k+1} + (\lambda^*)^\top r_h^{k+1}
  \]
\end{lemma}
\begin{proof}
  It follows from the definitions, vanishing constraints in an
  optimum, and Inequality~\ref{eq:assumption} that
  \begin{align*}
    f^*
    & = f_1(x^*) + f_2(z^*) \\
    & = L_0(x^*, z^*, \mu^*, \lambda^*) \\
    & \leq L_0(x^{k+1}, z^{k+1}, \mu^*, \lambda^*) \\
    & = f_1(x^{k+1}) + f_2(z^{k+1}) + (\mu^*)^\top r_g^{k+1} +
    (\lambda^*)^\top r_h^{k+1} \\
    & = f^{k+1} + (\mu^*)^\top r_g^{k+1} + (\lambda^*)^\top r_h^{k+1}.
  \end{align*}
\end{proof}

\begin{lemma} \label{lem:2}
  The difference between the value of the objective function at the
  $(k+1)$-th iterate and its optimal value can be bounded as follows
  \begin{align*}
    f^{k+1} - f^* \leq &-(\mu^{k+1})^\top r_g^{k+1} -
    (\lambda^{k+1})^\top r_h^{k+1} \\
    &- \rho\left(h_2(z^{k+1}) - h_2(z^{k})\right)^\top
    \left(-r_h^{k+1} + h_2(z^{k+1}) - h_2(z^*)\right).
  \end{align*}
\end{lemma}
\begin{proof}
  From Line~5 of Algorithm~1 we know that $x^{k+1}$ minimizes the
  function $L_\rho(x, z^k, \mu^k, \lambda^k)$ with respect to
  $x$. Hence, we know that $0$ must be contained in the
  subdifferential of $L_\rho(x, z^k, \mu^k, \lambda^k)$ with respect
  to $x$ at $x^{k+1}$, i.e.,
  \begin{align*}
    0\,\in &\: \partial f_1(x^{k+1}) + \rho \cdot \partial g(x^{k+1})
    \cdot g(x^{k+1}) + \partial g(x^{k+1}) \cdot \mu^k \\ &\: + \rho
    \cdot \partial h_1(x^{k+1}) \cdot \left(h_1(x^{k+1}) +
    h_2(z^k)\right) + \partial h_1(x^{k+1}) \cdot \lambda^k,
  \end{align*}
  where $\partial f_1(x^{k+1})\in\setR^{n_1}$ is the subdifferential
  of $f_1$ at $x^{k+1}$, $\partial g(x^{k+1})\in\setR^{n_1 \times p}$
  is the subdifferential of $g$ at $x^{k+1}$, and $\partial
  h_1(x^{k+1})\in\setR^{n_1 \times m}$ is the subdifferential of $h_1$
  at $x^{k+1}$.

  The update rule for the dual variables $\mu$ in Line~7 of
  Algorithm~1 gives
  \[
  \mu^k = \mu^{k+1} - \rho g(x^{k+1})
  \]
  and similarly, the update rule for the dual variables $\lambda$ in
  Line~8 gives
  \[
  \lambda^k = \lambda^{k+1} - \rho \left(h_1(x^{k+1}) + h_2(z^{k+1})\right).
  \]
  Plugging these update rules into the subdifferential optimality
  condition from above gives
  \begin{align*}
    0 \,\in &\: \partial f_1(x^{k+1}) + \rho \cdot \partial g(x^{k+1})
    \cdot g(x^{k+1}) + \partial g(x^{k+1}) \cdot \left(\mu^{k+1} -
    \rho g(x^{k+1}) \right) \\ & + \rho \cdot \partial h_1(x^{k+1})
    \cdot \left(h_1(x^{k+1}) + h_2(z^k)\right) + \partial h_1(x^{k+1})
    \cdot \left(\lambda^{k+1} - \rho h_1(x^{k+1}) - \rho
    h_2(z^{k+1})\right),
  \end{align*}
  and thus
  \[
  0\in \partial f_1(x^{k+1}) + \partial g(x^{k+1}) \cdot \mu^{k+1} +
  \partial h_1(x^{k+1}) \cdot \left(\lambda^{k+1} - \rho
  \left(h_2(z^{k+1}) - h_2(z^k)\right)\right).
  \]
  If $0$ is contained in the subdifferential of a convex function at
  point $x$, then $x$ is a minimizer of this function. That is,
  $x^{k+1}$ minimizes the convex function
  \begin{equation} \label{eq:minx}
    x \mapsto f_1(x) + (g(x))^\top \mu^{k+1} + (h_1(x))^\top
    \left(\lambda^{k+1} - \rho \left(h_2(z^{k+1}) -
    h_2(z^k)\right)\right).
  \end{equation}
  This function is convex, because $f_1$ and $g$ are convex functions,
  $h_1$ is an affine function, and any non-negative combination of
  convex functions is again a convex function. Note that we have
  $\mu^{k+1} \geq 0$ by Lemma~1.
  
  Similarly, Line~6 of Algorithm~1 implies that $0$ is contained in
  the subdifferential of $L_\rho(x^{k+1}, z, \mu^k, \lambda^k)$ with
  respect to $z$ at $z^{k+1}$, i.e.,
  \[
  0\in \partial f_2(z^{k+1}) + \rho \cdot \partial h_2(z^{k+1}) \cdot
  \left(h_1(x^{k+1}) + h_2(z^{k+1})\right) + \partial h_2(z^{k+1})
  \cdot \lambda^k.
  \]
  Again, substituting $\lambda^k = \lambda^{k+1} -
  \rho\left(h_1(x^{k+1}) + h_2(z^{k+1})\right)$ we get
  \[
  0 \in \partial f_2(z^{k+1}) + \partial h_2(z^{k+1}) \cdot
  \lambda^{k+1}.
  \]
  Hence, $z^{k+1}$ minimizes the convex function
  \begin{equation} \label{eq:minz}
    z \mapsto f_2(z) + (\lambda^{k+1})^\top h_2(z).
  \end{equation}
  The function is convex since $f_2$ is convex and $h_2$ is affine.
  
  Since $x^{k+1}$ is a minimizer of Function~\ref{eq:minx} we have
  \begin{align*}
    &f_1(x^{k+1}) + (g(x^{k+1}))^\top \mu^{k+1} + (h_1(x^{k+1}))^\top
    \left(\lambda^{k+1} - \rho \left(h_2(z^{k+1}) -
    h_2(z^k)\right)\right) \\
    &\quad \leq\: f_1(x^*) + (g(x^*))^\top \mu^{k+1} + (h_1(x^*))^\top
    \left(\lambda^{k+1} - \rho \left(h_2(z^{k+1}) -
    h_2(z^k)\right)\right).
  \end{align*}
  Analogously, since $z^{k+1}$ minimizes Function~\ref{eq:minz} we
  have
  \[
  f_2(z^{k+1}) + (\lambda^{k+1})^\top h_2(z^{k+1}) \leq f_2(z^*) +
  (\lambda^{k+1})^\top h_2(z^*).
  \]
  Finally, from summing up both inequalities and rearranging we get
  \begin{align*}
    &f^{k+1} - f^* \, =\, f_1(x^{k+1}) + f_2(z^{k+1}) - f_1(x^*) - f_2(z^*) \\
    &\leq (g(x^*))^\top \mu^{k+1} + (h_1(x^*))^\top
    \left(\lambda^{k+1} - \rho \left(h_2(z^{k+1}) - h_2(z^k)\right)\right) \\
    &\quad- (g(x^{k+1}))^\top \mu^{k+1} - (h_1(x^{k+1}))^\top
    \left(\lambda^{k+1} - \rho \left(h_2(z^{k+1}) - h_2(z^k)\right)\right) \\
    &\quad+ (\lambda^{k+1})^\top h_2(z^*) - (\lambda^{k+1})^\top
    h_2(z^{k+1}) \\
    &= - (g(x^{k+1}))^\top \mu^{k+1} +
    (\lambda^{k+1})^\top \left(h_1(x^*) + h_2(z^*)\right)
    -(\lambda^{k+1})^\top \left(h_1(x^{k+1}) + h_2(z^{k+1})\right) \\
    &\quad-\rho\left(h_2(z^{k+1}) -h_2(z^k)\right)^\top\left(h_1(x^*) -
    h_1(x^{k+1})\right) \\
    &= -(\mu^{k+1})^\top r_g^{k+1} -(\lambda^{k+1})^\top r_h^{k+1} \\
    &\quad - \rho\left(h_2(z^{k+1}) -h_2(z^k)\right)^\top\left(h_1(x^*) -
    h_1(x^{k+1})\right) \\
    &= -(\mu^{k+1})^\top r_g^{k+1} - (\lambda^{k+1})^\top r_h^{k+1} \\
    &\quad - \rho\left(h_2(z^{k+1}) -h_2(z^k)\right)^\top\left(h_1(x^*) +
    h_2(z^*) - h_2(z^*) - h_1(x^{k+1}) - h_2(z^{k+1}) +
    h_2(z^{k+1})\right) \\
    &= -(\mu^{k+1})^\top r_g^{k+1} - (\lambda^{k+1})^\top r_h^{k+1} \\
    &\quad -\rho\left(h_2(z^{k+1})
    -h_2(z^k)\right)^\top\left(-h_2(z^*) - r_h^{k+1} + h_2(z^{k+1})\right),
  \end{align*}
  where we have used that $g(x^*)= 0$, $h_1(x^*) + h_2(z^*) = 0$,
  $r^{k+1}_g = g(x^{k+1})$ is the residual of the convex constraints,
  and $r^{k+1}_h = h_1(x^{k+1}) + h_2(z^{k+1})$ is the residual of the
  affine constraints in iteration $k+1$.
\end{proof}

To continue, we need one more definition.

\begin{definition} \label{def:1}
  Let
  \[
  V^k =\:\, \frac{1}{\rho}\|\mu^k - \mu^*\|^2 +
  \frac{1}{\rho}\|\lambda^k - \lambda^*\|^2 \\
  + \rho \|h_2(z^k) - h_2(z^*)\|^2.
  \]
\end{definition}

For this newly defined quantity we show in the following lemma that it
is non-increasing over the iterations. This property will be crucial
for the proof of Theorem~\ref{thm:1}.

\begin{lemma} \label{lem:3}
  For every iteration $k\in\mathbb{N}$ it holds that
  \[
    \rho\|r_g^{k+1}\|^2 + \rho\|r_h^{k+1}\|^2 + \rho\|h_2(z^{k+1}) -
    h_2(z^k)\|^2 \leq V^k - V^{k+1}.
  \]
\end{lemma}
\begin{proof}
Summing up the inequality in Lemma~2 and the inequality in Lemma~3
gives
\begin{align*}
  0 \,\leq
  &\: (\mu^*)^\top r^{k+1}_g - (\mu^{k+1})^\top r^{k+1}_g +
  (\lambda^*)^\top r_h^{k+1} - (\lambda^{k+1})^\top r_h^{k+1} \\
  &\: - \rho\left(h_2(z^{k+1}) -
  h_2(z^{k})\right)^\top\left(-r_h^{k+1} + h_2(z^{k+1}) -
  h_2(z^*)\right),
\end{align*}
or equivalently, by rearranging and multiplying by $2$,
\begin{align} \label{eq:long}
  0 \,\geq
  &\: 2(\mu^{k+1} - \mu^*)^\top r_g^{k+1} +
  2(\lambda^{k+1}-\lambda^*)^\top r_h^{k+1} \nonumber \\ 
  &\: + 2\rho\left(h_2(z^{k+1}) -
  h_2(z^{k})\right)^\top\left(-r_h^{k+1} + h_2(z^{k+1}) -
  h_2(z^*)\right).
\end{align}
Next we are rewriting the three terms in this inequality individually.

Using the update rule $\mu^{k+1} = \mu^k + \rho g(x^{k+1}) = \mu^k +
\rho r_g^{k+1}$ for the Lagrange multipliers $\mu$ in Line~7 of
Algorithm~1 several times we can rewrite the first term as follows
\begin{align*}
  &2(\mu^{k+1}-\mu^*)^\top r_g^{k+1} \\
  &\quad =\: 2(\mu^k + \rho r_g^{k+1} -\mu^*)^\top r_g^{k+1} \\
  &\quad =\: 2(\mu^k -\mu^*)^\top r_g^{k+1} + \rho\|r_g^{k+1}\|^2
  +\rho\|r_g^{k+1}\|^2 \\
  &\quad =\: 2\frac{(\mu^k -\mu^*)^\top (\mu^{k+1} -\mu^k)}{\rho} +
  \frac{\|\mu^{k+1}-\mu^k\|^2}{\rho}  +\rho\|r_g^{k+1}\|^2 \\
  &\quad =\: 2\frac{(\mu^k -\mu^*)^\top
    \left(\mu^{k+1} - \mu^* -(\mu^k -\mu^*)\right)}{\rho} \\
  &\qquad + \frac{\|\mu^{k+1} - \mu^* -(\mu^k - \mu^*)\|^2}{\rho}
  +\rho\|r_g^{k+1}\|^2 \\
  &\quad =\: \frac{2(\mu^k -\mu^*)^\top \left(\mu^{k+1} - \mu^*\right) -
    2\|\mu^k -\mu^*\|^2}{\rho} \\
  &\qquad\: + \frac{\|\mu^{k+1} - \mu^*\|^2 + \|\mu^k - \mu^*\|^2
    - 2(\mu^{k+1} - \mu^*)^\top(\mu^k - \mu^*)}{\rho}  +\rho\|r_g^{k+1}\|^2 \\
  &\quad =\: \frac{\|\mu^{k+1} -\mu^*\|^2 - \|\mu^k - \mu^*\|^2}{\rho}
  +\rho\|r_g^{k+1}\|^2.
\end{align*}

The analogous argument holds for the second term, when using the
update rule $\lambda^{k+1} = \lambda^k + \rho r_h^{k+1}$ in Line~8 of
Algorithm~1, i.e., we have
\[
  2(\lambda^{k+1}-\lambda^*)^\top r_h^{k+1} \,=\,
  \frac{\|\lambda^{k+1} -\lambda^*\|^2 - \|\lambda^k - \lambda^*\|^2}
       {\rho} +\rho\|r_h^{k+1}\|^2.
\]

Adding $\rho\norm{r_h^{k+1}}^2$ to the third term of
Inequality~\ref{eq:long} gives
\begin{align*}
  &\rho\norm{r_h^{k+1}}^2 + 2\rho\left(h_2(z^{k+1}) -
  h_2(z^{k})\right)^\top\left(-r_h^{k+1} + h_2(z^{k+1}) -
  h_2(z^*)\right) \\
  & =\: \rho\|r_h^{k+1}\|^2 - 2\rho\left(h_2(z^{k+1}) -
  h_2(z^k)\right)^\top r_h^{k+1} + 2\rho\left(h_2(z^{k+1}) -
  h_2(z^k)\right)^\top\left(h_2(z^{k+1})-h_2(z^*)\right) \\  
  & =\: \rho\|r_h^{k+1}\|^2 - 2\rho\left(h_2(z^{k+1}) -
  h_2(z^k)\right)^\top r_h^{k+1} \\
  &\quad\:+ 2\rho\left(h_2(z^{k+1}) -
  h_2(z^k)\right)^\top\left(\left(h_2(z^{k+1})-h_2(z^k)\right) +
  \left(h_2(z^k) -h_2(z^*)\right)\right) \\
  & =\: \rho\|r_h^{k+1}\|^2 - 2\rho\left(h_2(z^{k+1}) -
  h_2(z^k)\right)^\top r_h^{k+1} \\
  &\quad\: + 2\rho\|h_2(z^{k+1}) - h_2(z^k)\|^2 +
  2\rho\left(h_2(z^{k+1}) - h_2(z^k)\right)^\top\left(h_2(z^k)
  -h_2(z^*)\right) \\
  & =\: \rho\|r_h^{k+1} - \left(h_2(z^{k+1}) - h_2(z^k)\right)\|^2 +
  \rho\|h_2(z^{k+1}) - h_2(z^k)\|^2 \\
  &\quad\: + 2\rho\left(h_2(z^{k+1}) -
  h_2(z^k)\right)^\top\left(h_2(z^k) -h_2(z^*)\right) \\
  & =\: \rho\|r_h^{k+1} - \left(h_2(z^{k+1}) - h_2(z^k)\right)\|^2 +
  \rho\|\left(h_2(z^{k+1}) - h_2(z^*)\right) - \left(h_2(z^k) -
  h_2(z^*)\right)\|^2 \\
  &\quad\: + 2\rho\left(\left(h_2(z^{k+1}) - h_2(z^*)\right)-
  \left(h_2(z^k)-h_2(z^*)\right)\right)^\top\left(h_2(z^k) -h_2(z^*)\right) \\
  & =\: \rho\|r_h^{k+1} - \left(h_2(z^{k+1}) - h_2(z^k)\right)\|^2 +
  \rho\|h_2(z^{k+1}) - h_2(z^*)\|^2 + \rho\|h_2(z^k) - h_2(z^*)\|^2 \\
  &\quad\: - 2\rho\left(h_2(z^{k+1}) -
  h_2(z^*)\right)^\top\left(h_2(z^k) - h_2(z^*)\right) \\
  &\quad\: + 2\rho\left(h_2(z^{k+1}) -
  h_2(z^*)\right)^\top\left(h_2(z^k) -h_2(z^*)\right) -
  2\rho \left(h_2(z^k)-h_2(z^*)\right)^\top\left(h_2(z^k) -h_2(z^*)\right) \\
  & =\: \rho\|r_h^{k+1} - \left(h_2(z^{k+1}) - h_2(z^k)\right)\|^2 +
  \rho\|h_2(z^{k+1}) - h_2(z^*)\|^2 - \rho\|h_2(z^k) - h_2(z^*)\|^2.
\end{align*}
Hence, Inequality~\ref{eq:long} is equivalent to
\begin{align*}
  0 \,\geq &\: \frac{1}{\rho}\left(\|(\mu^{k+1} -\mu^*\|^2 - \|\mu^k -
  \mu^*\|^2\right) + \rho\|r_g^{k+1}\|^2 \\
  &\:+ \frac{1}{\rho}\left(\|(\lambda^{k+1} -\lambda^*\|^2 -
  \|\lambda^k - \lambda^*\|^2\right) \\
  &\:+ \rho\|r_h^{k+1} - \left(h_2(z^{k+1}) - h_2(z^k)\right)\|^2 +
  \rho\|h_2(z^{k+1}) - h_2(z^*)\|^2 - \rho\|h_2(z^k) - h_2(z^*)\|^2.
\end{align*}
By rearranging the terms in this inequality and using the
following term expansion
\[
\|r_h^{k+1} - (h_2(z^{k+1}) - h_2(z^k))\|^2 = \|r_h^{k+1}\|^2 + \|h_2(z^{k+1}) -
h_2(z^k)\|^2 - 2(r_h^{k+1})^\top \left(h_2(z^{k+1}) - h_2(z^k)\right),
\]
we get
\begin{align*}
  &\rho\|r_g^{k+1}\|^2 + \rho\|r_h^{k+1}\|^2 + \rho\|h_2(z^{k+1}) - h_2(z^k)\|^2
  -2\rho(r_h^{k+1})^\top \left(h_2(z^{k+1}) - h_2(z^k)\right) \\
  &\quad =\:\rho\|r_g^{k+1}\|^2 + \rho\|r_h^{k+1} - \left(h_2(z^{k+1}) -
  h_2(z^k)\right)\|^2 \\
  &\quad \leq\: \frac{1}{\rho}\|\mu^k - \mu^*\|^2 +
  \frac{1}{\rho}\|\lambda^k - \lambda^*\|^2 + \rho \|h_2(z^k) -
  h_2(z^*)\|^2 \\
  &\qquad\: - \left(\frac{1}{\rho}\|\mu^{k+1} - \mu^*\|^2 +
  \frac{1}{\rho}\|\lambda^{k+1} - \lambda^*\|^2 + \rho\|h_2(z^{k+1}) -
  h_2(z^*)\|^2\right) \\
  &\quad =\: V^k - V^{k+1},
\end{align*}
where we have used Definition~1 of $V^k$ in the last equality. Hence,
to finish the proof of Lemma~\ref{lem:3} it only remains to show that
\[
2\rho(r_h^{k+1})^\top \left(h_2(z^{k+1}) - h_2(z^k)\right) \leq 0.
\]
From the proof of Lemma~\ref{lem:2} we know that $z^{k+1}$ minimizes
the function $f_2(z) + (\lambda^{k+1})^\top h_2(z)$ and similarly that
$z^k$ minimizes the function $f_2(z) + (\lambda^k)^\top
h_2(z)$. Hence, we have the following two inequalities
\[
f_2(z^{k+1}) +(\lambda^{k+1})^\top h_2(z^{k+1}) \leq f_2(z^{k})
+(\lambda^{k+1})^\top h_2(z^{k})
\]
and
\[
f_2(z^{k}) +(\lambda^{k})^\top h_2(z^{k}) \leq f_2(z^{k+1})
+(\lambda^{k})^\top h_2(z^{k+1}).
\]
Summing up these two inequalities yields
\[
(\lambda^{k+1})^\top h_2(z^{k+1}) + (\lambda^{k})^\top h_2(z^{k})
\leq (\lambda^{k+1})^\top h_2(z^{k}) + (\lambda^{k})^\top
h_2(z^{k+1}),
\]
or equivalently
\begin{align*}
  0 \,&\geq \: (\lambda^{k+1})^\top \left(h_2(z^{k+1}) -
  h_2(z^k)\right) + (\lambda^{k})^\top \left(h_2(z^{k}) -
  h_2(z^{k+1})\right) \\
  &=\: \left(\lambda^{k+1} - \lambda^k\right)^\top \left(h_2(z^{k+1})
  - h_2(z^k)\right) 
  \:=\: \rho (r_h^{k+1})^\top \left(h_2(z^{k+1}) - h_2(z^k)\right),
\end{align*}
where we have used the update rule for $\lambda$, see again Line~8 of
Algorithm~1. This completes the proof of Lemma~\ref{lem:3}.
\end{proof}

Now we are prepared to prove our main theorem.

\begin{proof}[Proof of Theorem~\ref{thm:1}]
  Using Lemma~\ref{lem:3} and that $V^k \geq 0$ for every iteration
  $k$, see Definition~\ref{def:1}, we can conclude that
  \[
    \rho \left( \sum_{k=0}^\infty \big( \|r_g^{k+1}\|^2 +
    \|r_h^{k+1}\|^2 + \|h_2(z^{k+1}) - h_2(z)\|^2 \big) \right) \leq
    \sum_{k=0}^\infty \big( V^k-V^{k+1} \big) \leq V^0.
  \]
  The series on the right hand side is absolutely convergent, because
  $V^0< \infty$, which follows from the fact that $h_2$ is an affine
  function. The absolute convergence implies
  \[
  \lim_{k\to\infty} r_g^k = 0, \quad 
  \lim_{k\to\infty} r_h^k = 0,
  \]
  and
  \[
  \lim_{k\to\infty} \|h_2(z^{k+1}) - h_2(z^k)\| = 0,
  \]
  i.e., the points $x^k$ and $z^k$ will be primal feasible up to an
  arbitrarily small error for sufficiently large $k$.  Finally, it
  follows from Lemmas~\ref{lem:1} and~\ref{lem:2} that
  $\lim_{k\to\infty} f^k = f^*$, i.e., the points $x^k$ and $z^k$ are
  also primal optimal up to an arbitrarily small error for
  sufficiently large $k$.
\end{proof}

\section{Convex optimization problems with many constraints}
\label{sec:general}

Finally, we are ready to discuss the main problem that we set out to
address in this paper, namely solving general convex optimization
problems with many constraints in a distributed setting by
distributing the constraints. That is, we want to address optimization
problems of the form
\begin{equation} \label{eq:prob1}
  \begin{array}{cl}
    \min_{x} & f(x) \\ \sthat & g_i(x) \leq 0 \quad i = 1\ldots p \\ &
    h_i(x) = 0\quad i = 1\ldots m,
  \end{array}
\end{equation}
where $f:\setR^n\to\setR$ and $g_i:\setR^n\to\setR^{p_i}$ are convex
functions, and $h_i:\setR^n\to\setR^{m_i}$ are affine functions. In
total, we have $p_1+p_2 + \ldots +p_p$ inequality constraints that are
grouped together into $p$ batches and $m_1+m_2+ \ldots + m_m$ equality
constraints that are subdivided into $m$ groups. For distributing the
constraints we can assume without loss of generality that $m=p$. That
is, we have $m$ batches that each contain $p_i$ inequality and $m_i$
equality constraints.

Again it is easier to work with an equivalent reformulation of
Problem~\ref{eq:prob1}, where each batch of equality and inequality
constraints shares the same variables $x_i$, namely problems of the
form
\begin{equation} \label{eq:prob2}
  \begin{array}{cll}
    \min_{x_i, z} & \sum_{i=1}^m f(x_i) \\
    \sthat & \max\{0,\, g_i(x_i)\}^2 = 0 & i = 1\ldots m \\
    & h_i(x_i) = 0 & i = 1\ldots m\\
    & x_i = z,
  \end{array}
\end{equation}
where all the variables $x_i$ are coupled through the affine
constraints $x_i=z$. To keep our exposition simple, the objective
function has been scaled by $m$ in the reformulation.

For specializing our extension of ADMM to instances of
Problem~\ref{eq:prob2} we need the Augmented Lagrangian of this
problem, which reads as
\begin{align*}
  L_\rho(x_i, z, \mu_{i,g}, \mu_{i,h}, \lambda) = 
  &\sum_{i=1}^m f(x_i) + \frac{\rho}{2}\sum_{i=1}^m \|\max\{0,\,  g_i(x_i)\}^2\|^2 \\
  &+ \sum_ i^m (\mu_{i, g})^\top \max\{0,\, g_i(x_i)\}^2 \\
  &+ \frac{\rho}{2}\sum_{i=1}^m\|h_i(x_i)\|^2 +
  \sum_ i^m (\mu_{i, h})^\top h_i(x_i) \\ 
  &+ \frac{\rho}{2}\sum_{i=1}^m \|x_i - z\|^2 + \sum_i^m (\lambda_i)^\top(x_i-z),
\end{align*}
where $\mu_{i, g}, \mu_{i, h}$, and $\lambda_i$ are the Lagrange
multipliers (dual variables).

Note that the Lagrange function is separable. Hence, the update of the
$x$ variables in Line~5 of Algorithm~\ref{algo:1} decomposes into the
following $m$ independent updates  
\begin{align*}
  x_i^{k+1} = \argmin_{x_i}\:\,
  & f(x_i) + \frac{\rho}{2}\|\max\{0,\, g_i(x_i)\}^2\|^2 \\
  & +(\mu_{i,g}^k)^\top \max\{0,\, g_i(x_i)\}^2 \\
  &+ \frac{\rho}{2}\|h_i(x_i)\|^2 + (\mu_{i, h}^k)^\top h_i(x_i) \\
  &+ \frac{\rho}{2}\|x_i-z^k\|^2 + (\lambda_i^k)^\top(x_i - z^k),
\end{align*}
that can be solved in parallel once the constraints $g_i(x_i)$ and
$h_i(x_i)$ have been distributed on $m$ different, distributed compute
nodes. Note that each update is an unconstrained, convex optimization
problem, because the functions that need to be minimized are sums of
convex functions. The only two summands where this might not be
obvious, are
\[
\frac{\rho}{2}\|\max\{0,\, g_i(x_i)\}^2\|^2
\]
and
\[
(\mu_{i,g}^k)^\top \max\{0,\, g_i(x_i)\}^2.
\]
For the first term note that the squared norm of a non-negative,
convex function is always convex again. The second term is convex,
because according to Lemma~\ref{lem:0} the $\mu_{i,g}^k$ are always
non-negative.

The update of the $z$ variable in Line~6 of Algorithm~\ref{algo:1}
amounts to solving the following unconstrained optimization problems
\begin{align*}
  z^{k+1} &= \argmin_z \sum_{i=1}^m \frac{\rho}{2}\|x_i^{k+1} -
  z\|^2 + \sum_{i=1}^m (\lambda_i^k)^\top(x_i^{k+1}-z) \\
  &=\frac{\rho\sum_{i=1}^m x_i^{k+1} + \sum_{i=1}^m \lambda_i^k}{\rho
    \cdot m},
\end{align*}
and the updates of the dual variables $\mu_i$ and $\lambda_i$ are as
follows
\begin{align*}
\mu_{i, g}^{k+1}  =\:\,&  \mu_{i, g}^k + \rho\, \max\{0,\, g_i(x_i^{k+1})\}^2, \\
\mu_{i, h}^{k+1}  =\:\,&  \mu_i^k + \rho\, h_i(x_i^{k+1}), \\
\lambda_i^{k+1}  =\:\,&  \lambda_i^k + \rho\left(x_i^{k+1} - z^{k+1}\right).
\end{align*}

That is, in each iteration there are $m$ independent, unconstrained
minimization problems that can be solved in parallel on different
compute nodes. The solutions of the independent subproblems are then
combined on a central node through the update of the $z$ variables and
the Lagrange multipliers. Actually, since the Lagrange multipliers
$\mu_{i, g}$ and $\mu_{i, h}$ are also local, i.e., involve only the
variables $x_i^{k+1}$ for any given index $i$, they can also be
updated in parallel on the same compute nodes where the $x_i^k$
updates take place. Only the variables $z$ and the Lagrange
multipliers $\lambda_i$ need to be updated centrally.

Looking at the update rules it becomes apparent that
Algorithm~\ref{algo:1} when applied to instances of
Problem~\ref{eq:prob2} is basically a combination of the standard
Augmented Lagrangian method~\cite{Hestenes69,Powell69} for solving
convex, constrained optimization problems and ADMM. It combines the
ability to solve constrained optimization problems (Augmented
Lagrangian) with the ability to solve convex optimization problems
distributedly (ADMM).

Let us briefly come back to the comparison with the alternative
approach of dealing with convex constraints $g_0(x)\leq 0$ by adding
appropriate indicator functions to the objective function and using
standard ADMM. As we have discussed before, every compute node has to
ensure feasibility and thus needs to project onto the feasible set
$\{x\,|\,g_0(x)\leq 0\}$ in every iteration. These projections are
quadratic, non-linearly constrained optimization problems. In contrast
to that, our extension of ADMM only needs to solve unconstrained
optimization problems in every iteration.

\newpage
\section{Experiments}
\label{sec:experiments}

We have implemented our extension of ADMM in Python using the NumPy
and SciPy libraries, and tested this implementation on the robust SVM
problem~\cite{ShivaswamyBS06} that has a second order cone constraint
for every data point. In our experiments we distributed these
constraints onto different compute nodes, where we had to solve an
unconstrained optimization problem in every iteration.

Since there is no other approach available that could deal with a
large number of arbitrary constraints in a distributed manner we
compare our approach to the baseline approach of running an Augmented
Lagrangian method in an outer loop and standard ADMM in an inner
loop. Note that this approach has three nested loops. The outer loop
turns the constrained problem into a sequence of unconstrained
problems (Augmented Lagrangians), the next loop distributes the
problem using distributed ADMM, and the final inner loop solves the
unconstrained subproblems using the L-BFGS-B
algorithm~\cite{MoralesN11,ZhuBLN97} in our implementation.

\subsection{Robust SVMs}

The robust SVM problem has been designed to deal with binary
classification problems whose input are not just labeled data points
$(x^{(1)},y^{(1)}),\ldots,(x^{(n)},y^{(n)})$, where the $x^{(i)}$ are
feature vectors and the $y^{(i)}$ are binary labels, but a
distribution over the feature vectors. That is, the labels are assumed
to be known precisely and the uncertainty is only in the features. The
idea behind the robust SVM is replacing the constraints (for feature
vectors without uncertainty)
of the standard linear soft-margin SVM by their probabilistic
counterparts
\[
\textsf{Pr} \big[ y^{(i)} w^\top x^{(i)} \geq 1 -\xi_i \big]
\geq 1 - \delta_i
\]
that require the now random variable $x^{(i)}$ with probability at
least $1-\delta_i \geq 0$ to be on the correct side of the hyperplane
whose normal vector is $w$.
Shivaswamy et al.\ show that the probabilistic constraints can be
written as second order cone constraints
\[
 y^{(i)} w^\top \bar x^{(i)} \geq 1 -\xi_i + \sqrt{\delta_i / (1-\delta_i)}
 \,\big\| \Sigma_i^{1/2} w\big\|,
\]
under the assumption that the mean of the random variable $x^{(i)}$ is
the empirical mean $\bar x^{(i)}$ and the covariance matrix of
$x^{(i)}$ is $\Sigma_i$.  The robust SVM problem is then the following
SOCP (second order cone program)
\begin{align*}
  &\min_{w,\xi}\: \frac{1}{2} \|w\|^2 + c \sum_{i=1}^n \xi_i \\
  &\sthat\:\, y^{(i)} w^\top \bar x^{(i)} \geq 1 -\xi_i +
  \sqrt{\delta_i / (1-\delta_i)} \big\| \Sigma_i^{1/2} w\big\| \\ 
  &\qquad\, \xi_i \geq 0, \quad\: i = 1\ldots n.
\end{align*}
This problem can be reformulated into the form of
Problem~\ref{eq:prob2} and is thus amenable to a distributed
implementation of our extension of ADMM.

\subsection{Experimental setup}

We generated random data sets similarly to~\cite{AndersonDLV12}, where
an interior point solver has been described for solving the robust SVM
problem. The set of feature vectors was sampled from a uniform
distribution on $[-1, 1]^n$. The covariance matrices $\Sigma_i$ were
randomly chosen from the cone of positive semidefinite matrices with
entries in the interval $[-1, 1]$ and $\delta_i$ has been set to
$\frac{1}{2}$. Each data point contributes exactly one constraint to
the problem and is assigned to only one of the compute nodes.

In the following, the primal optimization variables are $w$ and $\xi$,
the consensus variables for the primal optimization variables $w$ are
still denoted as $z$, and also the dual variables are still denoted as
$\lambda$ for the consensus constraints and $\mu$ for the convex
constraints, respectively.

\subsection{Convergence results}

\begin{figure*}[t!]
	\begin{center}
    \includegraphics[width=0.32\textwidth]{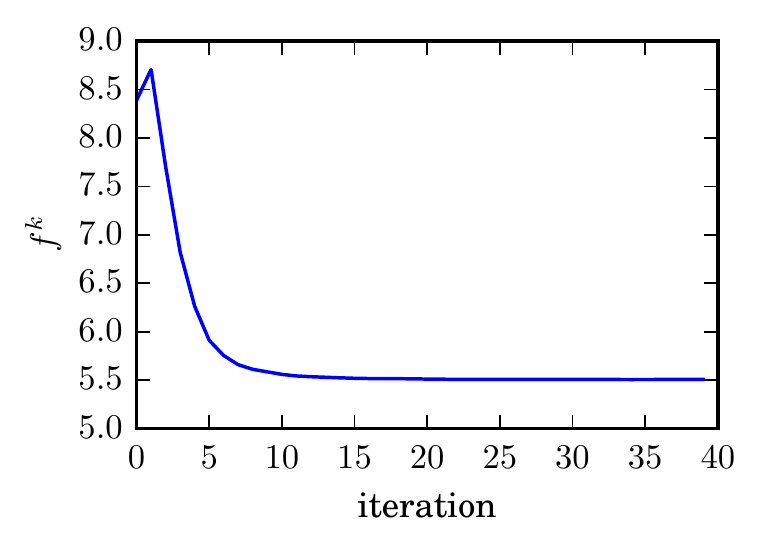}
    \includegraphics[width=0.32\textwidth]{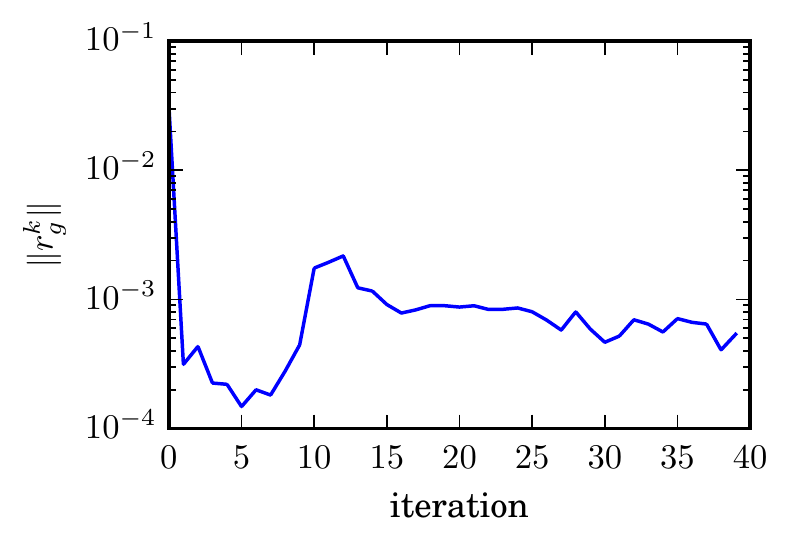}
    \includegraphics[width=0.32\textwidth]{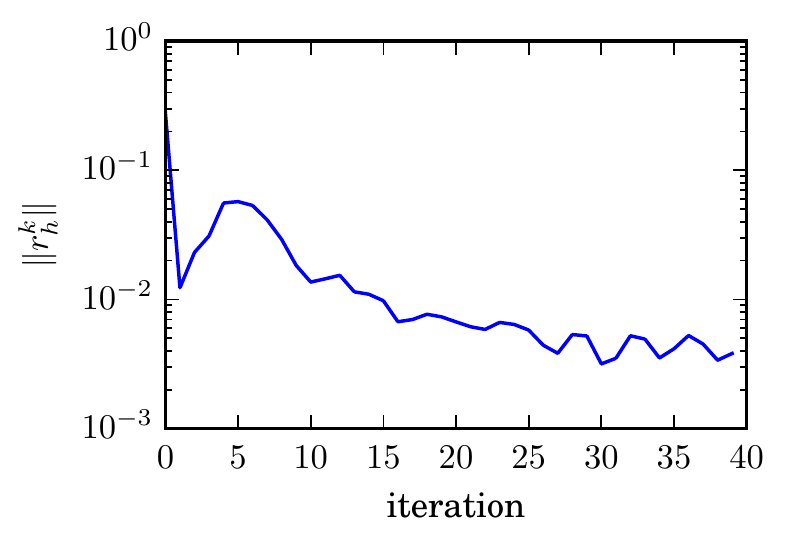}

    \includegraphics[width=0.32\textwidth]{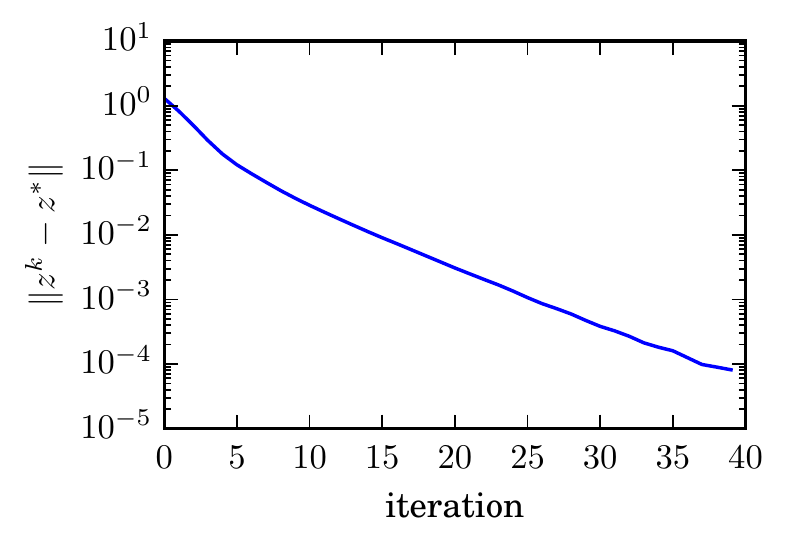}
    \includegraphics[width=0.32\textwidth]{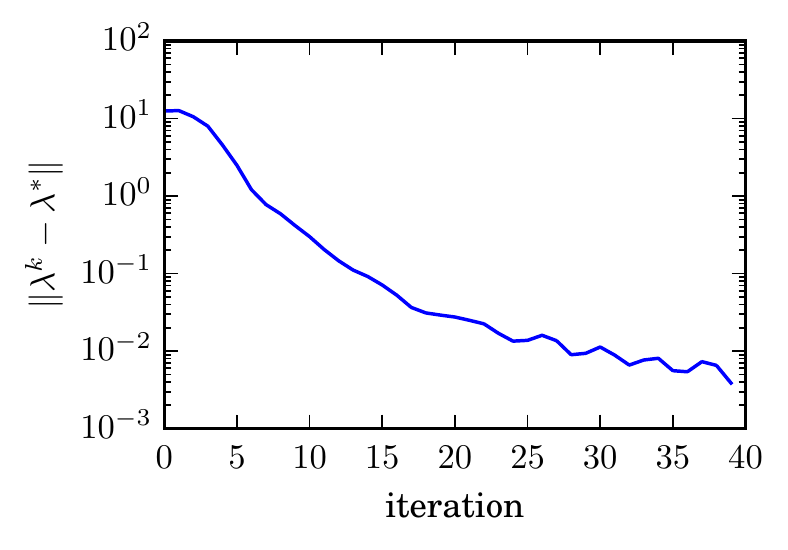}
    \includegraphics[width=0.32\textwidth]{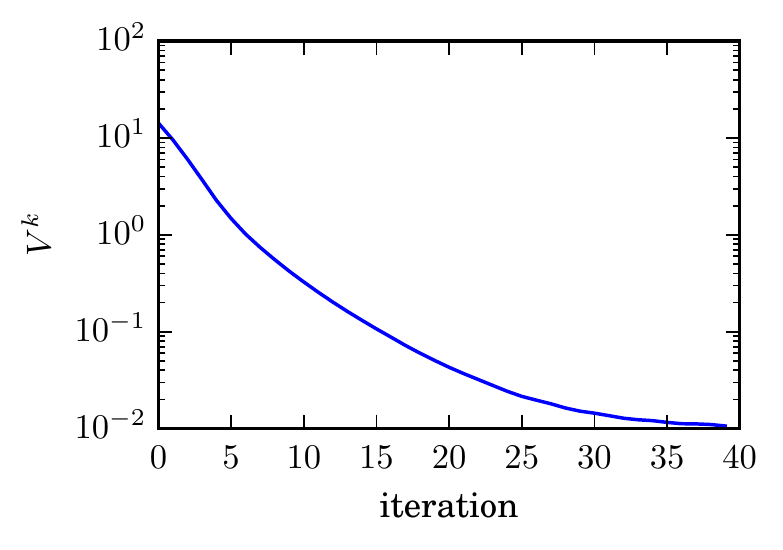}
	\end{center}
\caption{Various statistics for the performance of the distributed
  ADMM extension on an instance of the robust SVM problem. The
  convergence proof only states that the value $V^k$ must be
  monotonically decreasing. This can be observed also experimentally
  in the figure on the bottom right. Neither the primal function value
  nor the residuals need to be monotonically decreasing, and as can be
  seen in the figures on the top, they actually do not decrease
  monotonically.}
\label{fig:1}
\end{figure*}
Figure~\ref{fig:1} shows the primal objective function value $f^k$, the
norm of the residuals $r_g^k$ and $r_h^k$, the distances $\|z^k - z^*\|$,
$\|\lambda^k - \lambda^*\|$, and the value $V^k$ of one run of our
algorithm for two compute nodes. Note, that only $V^k$ must be
strictly monotonically decreasing according to our convergence
analysis. The proof does not make any statement about the monotonicity
of the other values, and as can be seen in Figure~\ref{fig:1}, such
statements would actually not be true. All values decrease in the long
run, but are not necessarily monotonically decreasing.

As can be seen in Figure~\ref{fig:1} (top-left), the function value
$f^k$ is actually increasing for the first few iterations, while the
residuals $r_g^k$ for the inequality constraints become very small,
see Figure~\ref{fig:1} (top-middle). That is, within the first
iterations each compute node finds a solution to its share of the data
that is almost feasible but has a higher function value than the true
optimal solution. This basically means that the errors $\xi_i$ for the
data points are over-estimated. After a few more iterations the primal
function value drops and the inequality residuals increase meaning
that the error terms $\xi_i$ as well as the individual estimators
$w_i$ converge to their optimal values.

In the long run, the local estimators at the different compute nodes
converge to the same solution. This is witnessed in Figure~\ref{fig:1}
(top-right), where one can see that the residuals $r_h^k$ for the
consensus constraints converge to zero, i.e., consensus among the
compute nodes is reached in the long run.

Finally, it can be seen that the consensus estimator $z^k$ converges
to its unique optimal point $z^*$. Note that in general we cannot
guarantee such a convergence since the optimal point does not need to
be unique. But of course, in the special case that the optimal point
is unique we always have convergence to this point.

\subsection{Scalability results}

\begin{figure*}[t!]
  \begin{center}
    \includegraphics[width=0.32\textwidth]{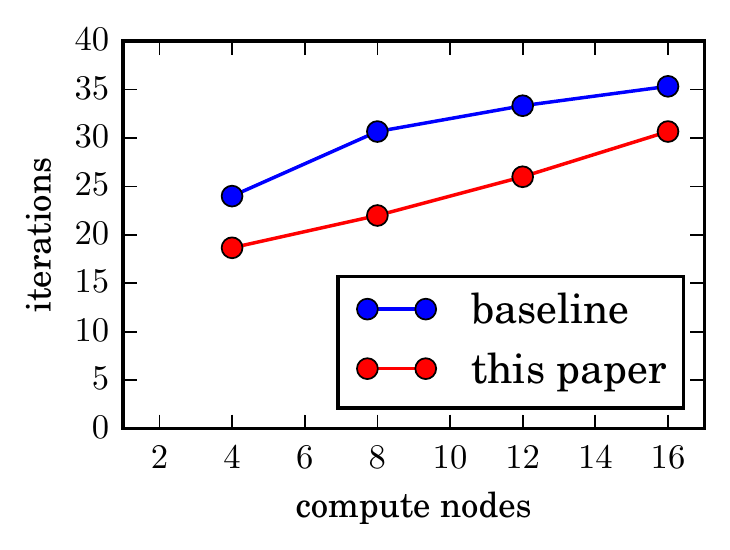} 
    \includegraphics[width=0.33\textwidth]{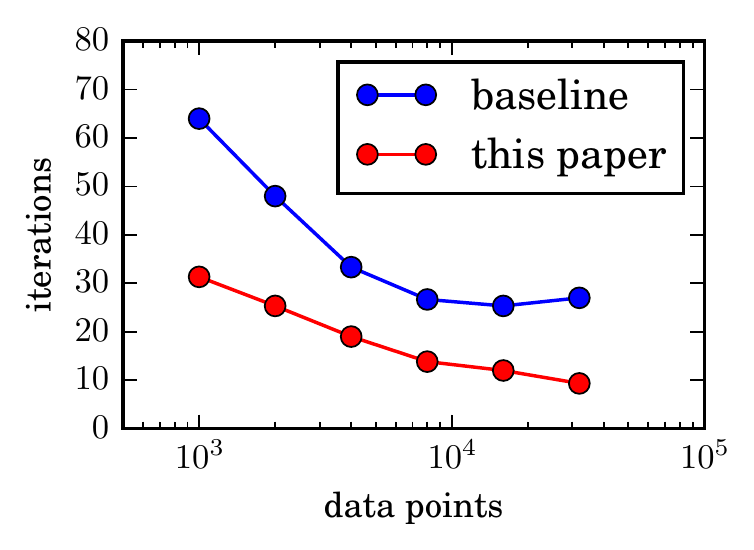}
    \includegraphics[width=0.32\textwidth]{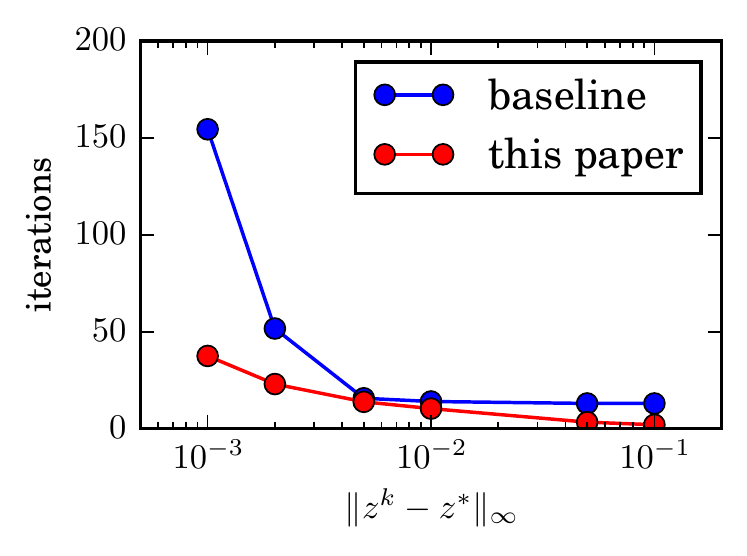}
  \end{center}
  \caption{Running times of the algorithm on the robust SVM
    problem. The figure on the left shows that the number of
    iterations increases mildly with the number of compute nodes. The
    middle picture shows that the number of iterations is decreasing
    with increasing number of data points. The figure on the right
    shows the dependency of the distance of the consensus estimator
    $z^k$ in iteration $k$ to the optimal estimator $z^*$. It can be
    seen that our extension of ADMM outperforms the baseline approach
    with three nested loops.}
\label{fig:2}
\end{figure*}
Figure~\ref{fig:2} shows the scalability of our extension of ADMM in
terms of the number of compute nodes, data points, and approximation
quality, respectively. All running times were measured in terms of
iterations and averaged over runs for ten randomly generated data
sets. The figures show the number of iterations for the approach
presented in this paper and for the baseline, i.e., a three nested
loops approach.

(1)\quad For measuring the scalability in terms of employed compute nodes,
we generated 10,000 data points with 10,000 features. As stopping
criterion we used
\[
\|z^k-z^*\|_\infty\leq 5\cdot 10^{-3},
\]
i.e., the predictor $z^k$ had to be close to the optimum. Here we use
the infinity norm to be independent from the number of dimensions. The
data set was split into four, eight, twelve, and 16 equal sized
batches that were distributed among the compute nodes. Note that every
batch had much fewer data points than features, and thus the optimal
solutions to the respective problems at the compute nodes were quite
different from each other. Nevertheless, our algorithm converged very
well to the globally optimal solution. Only the convergence speed was
affected by the diversity of the local solutions at the different
compute nodes. Since we kept the total number of data points in our
experiments fixed, the diversity was increasing with the number of
compute nodes that were assigned fewer data points each. Hence it was
expected that the convergence speed decreases, i.e., the number of
iterations increases, with a growing number of compute nodes. The
expected behavior can be seen in Figure~\ref{fig:2} (left). However,
the increase is rather mild. The number of iterations less than
doubles when the number of compute nodes increases from four to 16.

(2)\quad For measuring the scalability in terms of the number of data
points we increased the number of data points but kept the number of
features fixed at 200. The stopping criterion for our algorithm was
again
\[
\|z^k-z^*\|_\infty\leq 5\cdot 10^{-3}.
\]
We used eight compute nodes to compute the solutions. Again, the
points were distributed equally among the compute nodes. This time one
would expect a decreasing running time with an increasing number of
data points, because the number of data points per machine is
increasing and thus also the diversity of the local solutions at the
different compute nodes is decreasing. That is, with an increasing
number of data points it should take fewer iterations to reach an
approximate consensus about the global solution among the compute
nodes. The results of the experiment that are shown in
Figure~\ref{fig:2} (middle) confirm this expectation. The number of
iterations indeed decreases with a growing number of data points. It
has been noted before by Shalev-Shwartz and
Srebro~\cite{ShalevShwartzS08} that an increasing number of data
points can require less work for providing a good predictor. We
observe a similar phenomenon here.

(3)\quad For measuring the scalability in terms of the approximation
quality, we generated 8000 data points in 200 dimensions. Again, eight
compute nodes were used for the experiments whose results are shown in
Figure~\ref{fig:2} (right). As expected the number of iterations
(running time) increases with increasing approximation quality that
was again measured in terms of the infinity norm. In this paper we are
not providing a theoretical convergence rate analysis, which we leave
for future work, but the experimental results shown here already
provide some intuition on the dependency of the number of iterations
in terms of the approximation quality: It seems that our extension of
ADMM can solve problems to a medium accuracy within a reasonable
number of iterations, but higher accuracy requires a significant
increase in the number of iterations. Such a behavior is well known
for standard ADMM without constraints~\cite{boydADMM}. In the context
of our example application, robust SVMs, medium accuracy usually is
sufficient as often higher accuracy solutions do not provide better
predictors, a phenomenon that is also known as regularization by early
stopping.

\section{Conclusions}
\label{sec:conclusion}

We have introduced and analyzed an algorithm for solving general
convex optimization problems with many non-linear constraints in a
distributed setting. The algorithm is based on an extension of the
alternating direction method of multipliers (ADMM). Experiments on the
robust SVM problem corroborate our theoretical convergence analysis
and demonstrate the scalability of the approach in terms of the number
of compute nodes as well as the number of data points.

Despite the vast literature on ADMM, to the best of our knowledge, an
ADMM-like scheme for distributing general convex constraints has not
been studied before. Standard ADMM is typically used for solving
unconstrained optimization problems with separable objective function
in a distributed fashion, but in principle standard ADMM can also be
used for solving constrained optimization problems. This leads, in the
distributed implementation of ADMM, to local, constrained optimization
problems that have to be solved in every iteration. These local
constrained optimization problems are easy to solve in special cases,
for instance for linear constraints, but can become hard to solve in
the general case of convex, non-linear constraints. In general three
nested loops are necessary in this approach, an outer loop for
reaching consensus, one loop for the constraints, and an inner loop
for solving unconstrained problems. Alternatively, one can use the
Augmented Lagrangian method for constrained optimization in the outer
loop and standard ADMM in the inner loop. This approach also entails
three nested loops, an outer loop for the constraints, one loop for
reaching consensus, and an inner loop for solving unconstrained
problems. That is, the tasks of the two outer loops, reaching
consensus and dealing with the constraints, are interchanged in the
two approaches. Here, we use the second approach, i.e., the Augmented
Lagrangian with ADMM in the inner loop, as our baseline since it
avoids the need for solving constrained problems in the inner
loop. But our main contribution is showing that the two loops for
reaching consensus and for handling constraints can be merged. This
results in an extension of ADMM for dealing with problems with many,
non-linear constraints in a distributed fashion that only needs two
nested loops. To the best of our knowledge, we provide the first
convergence proof for such a lazy algorithmic scheme. Experimental
results provide evidence that our two-loop algorithm is indeed more
efficient than the baseline approach with three nested loops.

\subsection*{Acknowledgments} This work was supported by Deutsche
Forschungsgemeinschaft (DFG) under grant GI-711/5-1 and grant LA2971/1-1.

\bibliographystyle{plain}

\end{document}